\documentclass[11pt]{amsart}
\usepackage{bbm,amsfonts,latexsym,amsmath,amscd,amssymb,fancyhdr}
\usepackage[all]{xy}
\usepackage{fullpage}

\theoremstyle{plain}
\newtheorem{theorem}{Theorem}
\numberwithin{theorem}{section}

\newtheorem{corollary}{Corollary}

\newtheorem{definition}{Definition}
\numberwithin{definition}{section}

\newtheorem{lemma}{Lemma}
\numberwithin{lemma}{section}

\newtheorem{proposition}{Proposition}
\numberwithin{proposition}{section}

\newtheorem{remark}{Remark}
\numberwithin{remark}{section}

\numberwithin{example}{section}

\numberwithin{equation}{section}

\newcommand {\be}{\begin{equation}}
\newcommand {\ee}{\end{equation}}

\newcommand{\h}{\begin{eqnarray*}}
\newcommand{\e}{\end{eqnarray*}}

\newcommand{\CC}{\mathbf{C}}

\newcommand{\ZZ}{\mathbf{Z}}

\newcommand{\ii}{\sqrt{-1}}

\begin{document}

\title[Anomaly Cancellation and Modularity]{Anomaly Cancellation and Modularity}
\author{Fei Han}
\address{Fei Han, Department of Mathematics, National University of Singapore,
 Block S17, 10 Lower Kent Ridge Road,
Singapore 119076 (mathanf@nus.edu.sg)}
\author{Kefeng Liu}
\address{Kefeng Liu, Department of Mathematics, University of California at Los Angeles,
Los Angeles, CA 90095, USA (liu@math.ucla.edu) and  Center of Mathematical Sciences, Zhejiang University, 310027, P.R. China}
\author{Weiping Zhang}
\address{Weiping Zhang, Chern Institute of Mathematics \& LPMC, Nankai
University, Tianjin 300071, P.R. China. (weiping@nankai.edu.cn)}
\maketitle

\begin{abstract}It has been shown that the Alvarez-Gaum$\mathrm{\acute{e}}$-Witten miraculous
anomaly cancellation formula in type IIB superstring theory and
its various generalizations can be derived from modularity of
certain characteristic forms. In this paper, we show that the
Green-Schwarz formula and the Schwarz-Witten formula in type I
superstring theory can also be derived from the modularity of
those characteristic forms and thus unify the
Alvarez-Gaum$\mathrm{\acute{e}}$-Witten formula, the Green-Schwarz
formula as well as the Schwarz-Witten formula in the same
framework. Various generalizations of these remarkable formulas
are also established.

\end{abstract}

\section*{Introduction} Let $Z\to X\to B$ be a fiber bundle with fiber $Z$ being 10 dimensional.
Let $TZ$ be the vertical tangent bundle equipped with a metric
$g^{TZ}$ and an associated Levi-Civita connection $\nabla^{TZ}$
(cf. \cite[Proposition 10.2]{BGV}). Let  $R^{TZ}=(\nabla^{TZ})^2$
be the curvature of $\nabla^{TZ}$. Let $T_\CC Z$ be the
complexification of $TZ$ with the induced Hermitian connection
$\nabla^{T_\CC Z}$. Let $\widehat{A}(TZ, \nabla^{TZ})$,
 $L(TZ, \nabla^{TZ})
$ and ${\rm ch}(T_\CC Z,\nabla^{T_\CC Z})$ be the Hirzebruch
$\widehat{A}$-form, the Hirzebruch $L$-form and the Chern
character form respectively (cf. \cite{BGV} and \cite{Z}).

The Alvarez-Gaum$\mathrm{\acute{e}}$-Witten ``miraculous anomaly
cancellation formula" \cite{AW} in type IIB superstring theory
  asserts that
 \be \{L(TZ, \nabla^{TZ})\}^{(12)}-8\{\widehat{A}(TZ,
\nabla^{TZ})\mathrm{ch}(T_\CC Z, \nabla^{T_\CC
Z})\}^{(12)}+16\{\widehat{A}(TZ, \nabla^{TZ})\}^{(12)}=0,\ee which
assures that the corresponding theory is anomaly-free.

On the other hand, Green and Schwarz (\cite{GS}, see also
\cite{S}) discovered that the anomaly in type I superstring theory
with gauge group $SO(32)$ cancels. They found that when the gauge
group is $SO(32)$, the anomaly factorizes so that there is a
Chern-Simons counterterm making the anomaly cancelled. More
precisely, let $F$ be a $32$ dimensional Euclidean vector bundle
over $X$ with Euclidean  connection $\nabla^F$ and $F_\CC$ the
complexfication of $F$ with the induced Hermitian connection
$\nabla^{F_\CC}$, then the Green-Schwarz formula reads as
follows,\footnote{In what follows, we will write characteristic
forms without specifying the connections when there is no
confusion.}
\be \begin{split} &\{\widehat{A}(TZ)\mathrm{ch}(\wedge^2F_\CC)\}^{(12)}+\{\widehat{A}(TZ)\mathrm{ch}(T_\CC Z)\}^{(12)}-2\{\widehat{A}(TZ)\}^{(12)}\\
=&(p_1(TZ)-p_1(F))\cdot \frac{1}{24}\left(\frac{-3p_1(TZ)^2+4p_2(TZ)}{8}-2p_1(F)^2+4p_2(F)+\frac{1}{2}p_1(TZ)p_1(F)\right),
\end{split}\ee
where $p_i(TZ),\, p_i(F), 1\leq i\leq 2$, are the Pontryagin forms
of $(TZ, \nabla^{TZ}), (F, \nabla^F)$ respectively.

The above formulas of Alvarez-Gaum$\mathrm{\acute{e}}$-Witten
 and   Green-Schwarz
  have played crucial roles in the early development of superstring
theory.

More recently,  Schwarz and Witten   \cite{SW} analyzed the
anomaly in type I theory with additional spacetime-filling
D-branes and anti-D-branes pairs and found a similar
factorization. More precisely, let $F_1$ be an $m$ dimensional
Euclidean vector bundle over $X$ equipped with a Euclidean
connection $\nabla^{F_1}$ and $F_2$ be an $n$ dimensional
Euclidean vector bundle over $X$ equipped with a Euclidean
connection $\nabla^{F_2}$, when $m=n+32$, one has \be
\begin{split}
&\{\widehat{A}(TZ)\mathrm{ch}(\wedge^2{F_1}_\CC)\}^{(12)}+
\{\widehat{A}(TZ)\mathrm{ch}(S^2{F_2}_\CC)\}^{(12)}-\{\widehat{A}(TZ)\mathrm{ch}({F_1}_\CC\otimes {F_2}_\CC)\}^{(12)}\\
&+\{\widehat{A}(TZ)\mathrm{ch}(T_\CC Z)\}^{(12)}-2\{\widehat{A}(TZ)\}^{(12)}\\
=&\left(p_1(TZ)-p_1(F_1)+p_1(F_2)\right)\\
&\cdot
\frac{1}{24}\left(\frac{-3p_1(TZ)^2+4p_2(TZ)}{8}-2p_1(F_1)^2+4p_2(F_1)+2p_1(F_2)^2-4p_2(F_2)+
\frac{1}{2}p_1(TZ)(p_1(F_1)-p_1(F_2))\right).
\end{split}\ee
Therefore, similarly, there is a Chern-Simons counterterm to make
the anomaly cancelled.

In \cite{Liu1, HL}, it is shown that the Alvarez-Gaum$\mathrm{\acute{e}}$-Witten ``miraculous anomaly
 cancellation formula" can be derived from the modularity of certain characteristic forms. In fact, let $V$ be a Euclidean
 vector bundle equipped with a Euclidean  connection over $X$,  one can construct two characteristic
  forms $P_1(TZ, V, \tau)$ and $P_2(TZ, V, \tau)$ such that when $p_1(TZ)=p_1(V)$, $P_1(TZ, V, \tau)$ and $P_2(TZ, V, \tau)$ are
  level 2 modular forms over $\Gamma_0(2)$ and $\Gamma^0(2)$ respectively.  Moreover they are modularly related and form what we call
   a modular pair (see page 9 for more details).
  The Alvarez-Gaum$\mathrm{\acute{e}}$-Witten   formula  can then  be
  deduced from this modular pair $(P_1(TZ, V, \tau), P_2(TZ, V, \tau))$ if one sets $V=TZ$. This construction
  is further
  generalized in \cite{HZ2}
  to the case where a complex line bundle is involved, in dealing with   the Ochanine congruence \cite{Och} on spin$^c$   manifolds.

 In a recent article \cite{HLZ}, in using the Eisenstein series $E_2(\tau)$, we constructed a pair of
 modularly related characteristic forms
 $(\mathcal{P}_1(TZ, V, \xi, \tau), \mathcal{P}_2(TZ, V, \xi, \tau))$
  without assuming  $p_1(TZ)=p_1(V)$. When  $p_1(TZ)=p_1(V)$ and $\xi$ is trivial,
 $(\mathcal{P}_1(TZ, V, \xi, \tau), \mathcal{P}_2(TZ, V, \xi, \tau))$ degenerates to $(P_1(TZ, V, \tau), P_2(TZ, V, \tau))$.

In the current  paper, we will show that the formulas due to
Green-Schwarz   (0.3) and Schwarz-Witten   (0.4) can also be
deduced from the modularity of the pair $(\mathcal{P}_1(TZ, V,
\xi, \tau), \mathcal{P}_2(TZ, V, \xi, \tau))$. Actually, we need
only to make use of the modularity of $\mathcal{P}_2(TZ, V, \xi,
\tau)$  by replacing  $V$ by a super vector bundle $F_1-F_2$. Our
method also generates many generalizations of the Green-Schwarz
  and   Schwarz-Witten formulas. See Theorem 1.1 and its
corollaries for more details.

It is quite  amazing   that  all of the three anomaly cancellation
formulas
 due to Alvarez-Gaum$\mathrm{\acute{e}}$-Witten,
  Green-Schwarz,
  as well as  Schwarz-Witten, can be unified through  a single
 modular pair $(\mathcal{P}_1(TZ, V, \xi, \tau),
\mathcal{P}_2(TZ, V, \xi, \tau))$.  It illustrates one of the deep
implications of modularity in physics.

In the rest of this paper, we will first present the Green-Schwarz
type factorization formulas in Section 1 and then show how to
derive them  from modularity in Section 2.

\section{Green-Schwarz Type Factorization Formulas}
The purpose of this section is to present various generalizations of the Green-Schwarz formula and the Schwarz-Witten formula.

Let $Z\to X\to B$ be a fiber bundle with fiber $Z$ being 10
dimensional. Let $TZ$ be the vertical tangent bundle equipped with
a metric $g^{TZ}$ and an associated Levi-Civita connection
$\nabla^{TZ}$ (cf. \cite[Proposition 10.2]{BGV}). Let
$R^{TZ}=(\nabla^{TZ})^2$ be the curvature of $\nabla^{TZ}$. Let
$T_\CC Z$ be the complexification of $TZ$ with the induced
Hermitian connection $\nabla^{T_\CC Z}$.

Let $F_1$ be an $m$ dimensional Euclidean vector bundle over $X$
equipped  with a Euclidean connection $\nabla^{F_1}$ and $F_2$ be
an $n$ dimensional Euclidean vector bundle over $X$ equipped with
a Euclidean connection $\nabla^{F_2}$.

Let $\xi$ be a rank two real oriented Euclidean vector bundle over
$X$ carrying a Euclidean connection $\nabla^{\xi}$. Let $c=e(\xi,
\nabla^\xi)$ be the Euler form canonically associated to
$\nabla^\xi$.

If $E$ is a real (resp. complex) vector bundle over $X$, set
$\widetilde{E}=E-{\dim E}\in KO(X)$ (resp.    $K(X)$).

If $\omega$ is a differential form, denote the degree $j$-component of $\omega$ by $\omega^{(j)}$.

\begin{theorem} The following identity holds,
\be \begin{split} &\{\widehat{A}(TZ)e^{\frac{c}{2}}\mathrm{ch}(\wedge^2{F_1}_\CC)\}^{(12)}+
\{\widehat{A}(TZ)e^{\frac{c}{2}}\mathrm{ch}(S^2{F_2}_\CC)\}^{(12)}-\{\widehat{A}(TZ)e^{\frac{c}{2}}\mathrm{ch}({F_1}_\CC\otimes {F_2}_\CC)\}^{(12)}\\
&+\{\widehat{A}(TZ)e^{\frac{c}{2}}\mathrm{ch}(T_\CC Z)\}^{(12)}+\left(\frac{(m-n-32)(m-n-31)}{2}-2\right)\{\widehat{A}(TZ)e^{\frac{c}{2}}\}^{(12)}\\
&-(m-n-32)\{\widehat{A}(TZ)e^{\frac{c}{2}}\mathrm{ch}({F_1}_\CC-{F_2}_\CC)\}^{(12)}\\
&+5\{\widehat{A}(TZ)e^{\frac{c}{2}}\mathrm{ch}(\widetilde{\xi_\CC}\otimes\widetilde{\xi_\CC})\}^{(12)}
+3\{\widehat{A}(TZ)e^{\frac{c}{2}}\mathrm{ch}((m-n-31-{F_1}_\CC+{F_2}_\CC)\otimes\widetilde{\xi_\CC})\}^{(12)}\\
=&(p_1(TZ)-p_1(F_1)+p_1(F_2))\\
&\cdot\left\{-\frac{e^{\frac{1}{24}(p_1(TZ)-p_1(F_1)+p_1(F_2))}-1}{p_1(TZ)-p_1(F_1)+p_1(F_2)}
\widehat{A}(TZ)e^{\frac{c}{2}}\mathrm{ch}(\mathfrak{A})+
e^{\frac{1}{24}(p_1(TZ)-p_1(F_1)+p_1(F_2))}\widehat{A}(TZ)e^{\frac{c}{2}}\right\}^{(8)},
\end{split}\ee
where
\be \begin{split} \mathfrak{A}=&\wedge^2{F_1}_\CC+S^2{F_2}_\CC-{F_1}_\CC\otimes {F_2}_\CC+T_\CC Z+
\frac{(m-n-32)(m-n-31)}{2}-2\\
&-(m-n-32)({F_1}_\CC-{F_2}_\CC)+5\widetilde{\xi_\CC}\otimes\widetilde{\xi_\CC}+
3(m-n-31-{F_1}_\CC+{F_2}_\CC)\otimes\widetilde{\xi_\CC};
\end{split} \ee
if  $\xi$ is trivial, the following identity holds,
\be \begin{split} &\{\widehat{A}(TZ)\mathrm{ch}(\wedge^2{F_1}_\CC)\}^{(12)}+
\{\widehat{A}(TZ)\mathrm{ch}(S^2{F_2}_\CC)\}^{(12)}-\{\widehat{A}(TZ)\mathrm{ch}({F_1}_\CC\otimes {F_2}_\CC)\}^{(12)}\\
&+\{\widehat{A}(TZ)\mathrm{ch}(T_\CC Z)\}^{(12)}+\left(\frac{(m-n-32)(m-n-31)}{2}-2\right)\{\widehat{A}(TZ)\}^{(12)}\\
&-(m-n-32)\{\widehat{A}(TZ)\mathrm{ch}({F_1}_\CC-{F_2}_\CC)\}^{(12)}\\
=&(p_1(TZ)-p_1(F_1)+p_1(F_2))\\
&\cdot\left\{-\frac{e^{\frac{1}{24}(p_1(TZ)-p_1(F_1)+p_1(F_2))}-1}{p_1(TZ)-p_1(F_1)+p_1(F_2)}
\widehat{A}(TZ)\mathrm{ch}(\mathfrak{B})+
e^{\frac{1}{24}(p_1(TZ)-p_1(F_1)+p_1(F_2))}\widehat{A}(TZ)\right\}^{(8)},
\end{split}\ee
where
\be \begin{split} \mathfrak{B}=&\wedge^2{F_1}_\CC+S^2{F_2}_\CC-{F_1}_\CC\otimes {F_2}_\CC+T_\CC Z+
\frac{(m-n-32)(m-n-31)}{2}-2\\
&-(m-n-32)({F_1}_\CC-{F_2}_\CC).
\end{split} \ee
\end{theorem}

Putting  $m=n+32$ in Theorem 1, we get
\begin{corollary}If $\dim F_1-\dim F_2=32$, the following identity holds,
\be \begin{split} &\{\widehat{A}(TZ)e^{\frac{c}{2}}\mathrm{ch}(\wedge^2{F_1}_\CC)\}^{(12)}+
\{\widehat{A}(TZ)e^{\frac{c}{2}}\mathrm{ch}(S^2{F_2}_\CC)\}^{(12)}-\{\widehat{A}(TZ)e^{\frac{c}{2}}\mathrm{ch}({F_1}_\CC\otimes {F_2}_\CC)\}^{(12)}\\
&+\{\widehat{A}(TZ)e^{\frac{c}{2}}\mathrm{ch}(T_\CC Z)\}^{(12)}-2\{\widehat{A}(TZ)e^{\frac{c}{2}}\}^{(12)}\\
&+5\{\widehat{A}(TZ)e^{\frac{c}{2}}\mathrm{ch}(\widetilde{\xi_\CC}\otimes\widetilde{\xi_\CC})\}^{(12)}
+3\{\widehat{A}(TZ)e^{\frac{c}{2}}\mathrm{ch}((1-{F_1}_\CC+{F_2}_\CC)\otimes\widetilde{\xi_\CC})\}^{(12)}\\
=&(p_1(TZ)-p_1(F_1)+p_1(F_2))\\
&\cdot\left\{-\frac{e^{\frac{1}{24}(p_1(TZ)-p_1(F_1)+p_1(F_2))}-1}{p_1(TZ)-p_1(F_1)+p_1(F_2)}
\widehat{A}(TZ)e^{\frac{c}{2}}\mathrm{ch}(\mathfrak{C})+
e^{\frac{1}{24}(p_1(TZ)-p_1(F_1)+p_1(F_2))}\widehat{A}(TZ)e^{\frac{c}{2}}\right\}^{(8)},
\end{split}\ee
where \be \begin{split} \mathfrak{C}=&\wedge^2{F_1}_\CC+S^2{F_2}_\CC-{F_1}_\CC\otimes {F_2}_\CC+T_\CC Z-2\\
&+5\widetilde{\xi_\CC}\otimes\widetilde{\xi_\CC}+3(1-{F_1}_\CC+{F_2}_\CC)\otimes\widetilde{\xi_\CC};
\end{split} \ee
if  $\xi$ is trivial, we obtain the Schwarz-Witten formula (0.4),
\be \begin{split} &\{\widehat{A}(TZ)\mathrm{ch}(\wedge^2{F_1}_\CC)\}^{(12)}+
\{\widehat{A}(TZ)\mathrm{ch}(S^2{F_2}_\CC)\}^{(12)}-\{\widehat{A}(TZ)\mathrm{ch}({F_1}_\CC\otimes {F_2}_\CC)\}^{(12)}\\
&+\{\widehat{A}(TZ)\mathrm{ch}(T_\CC Z)\}^{(12)}-2\{\widehat{A}(TZ)\}^{(12)}\\
=&(p_1(TZ)-p_1(F_1)+p_1(F_2))\\
&\cdot\left\{-\frac{e^{\frac{1}{24}(p_1(TZ)-p_1(F_1)+p_1(F_2))}-1}{p_1(TZ)-p_1(F_1)+p_1(F_2)}
\widehat{A}(TZ)\mathrm{ch}(\mathfrak{D})+
e^{\frac{1}{24}(p_1(TZ)-p_1(F_1)+p_1(F_2))}\widehat{A}(TZ)\right\}^{(8)},
\end{split}\ee
where \be \begin{split} \mathfrak{D}=&\wedge^2{F_1}_\CC+S^2{F_2}_\CC-{F_1}_\CC\otimes {F_2}_\CC+T_\CC Z-2.
\end{split} \ee
\end{corollary}

\begin{remark} It can be checked by direct computation that when $m=n+32$, one indeed has
 \be \begin{split}
& \frac{1}{24}\left(\frac{-3p_1(TZ)^2+4p_2(TZ)}{8}-2p_1(F_1)^2+4p_2(F_1)+2p_1(F_2)^2-4p_2(F_2)+
\frac{1}{2}p_1(TZ)(p_1(F_1)-p_1(F_2)\right)\\
=&\left\{-\frac{e^{\frac{1}{24}(p_1(TZ)-p_1(F_1)+p_1(F_2))}-1}{p_1(TZ)-p_1(F_1)+p_1(F_2)}
\widehat{A}(TZ)\mathrm{ch}(\mathfrak{D})+
e^{\frac{1}{24}(p_1(TZ)-p_1(F_1)+p_1(F_2))}\widehat{A}(TZ)\right\}^{(8)}.
\end{split} \ee
\end{remark}

If we set $n=0$ in Theorem 1,   we get
\begin{corollary}If $\dim F=m$, then the following identity holds,
\be \begin{split} &\{\widehat{A}(TZ)e^{\frac{c}{2}}\mathrm{ch}(\wedge^2{F}_\CC)\}^{(12)}+
\{\widehat{A}(TZ)e^{\frac{c}{2}}\mathrm{ch}(T_\CC Z)\}^{(12)}+\left(\frac{(m-32)(m-31)}{2}-2\right)\{\widehat{A}(TZ)e^{\frac{c}{2}}\}^{(12)}\\
&-(m-32)\{\widehat{A}(TZ)e^{\frac{c}{2}}\mathrm{ch}({F}_\CC)\}^{(12)}\\
&+5\{\widehat{A}(TZ)e^{\frac{c}{2}}\mathrm{ch}(\widetilde{\xi_\CC}\otimes\widetilde{\xi_\CC})\}^{(12)}
+3\{\widehat{A}(TZ)e^{\frac{c}{2}}\mathrm{ch}((m-31-{F}_\CC)\otimes\widetilde{\xi_\CC})\}^{(12)}\\
=&(p_1(TZ)-p_1(F))\\
&\cdot\left\{-\frac{e^{\frac{1}{24}(p_1(TZ)-p_1(F))}-1}{p_1(TZ)-p_1(F)}
\widehat{A}(TZ)e^{\frac{c}{2}}\mathrm{ch}(\mathfrak{E})+
e^{\frac{1}{24}(p_1(TZ)-p_1(F))}\widehat{A}(TZ)e^{\frac{c}{2}}\right\}^{(8)},
\end{split}\ee
where
\be \begin{split} \mathfrak{E}=&\wedge^2{F_1}_\CC+T_\CC Z+
\frac{(m-32)(m-31)}{2}-2\\
&-(m-32)({F}_\CC)+5\widetilde{\xi_\CC}\otimes\widetilde{\xi_\CC}+
3(m-31-{F}_\CC)\otimes\widetilde{\xi_\CC};
\end{split} \ee
if $\xi$ is trivial, the following identity holds,
\be \begin{split} &\{\widehat{A}(TZ)\mathrm{ch}(\wedge^2{F}_\CC)\}^{(12)}+\{\widehat{A}(TZ)\mathrm{ch}(T_\CC Z\}^{(12)}+\left(\frac{(m-32)(m-31)}{2}-2\right)\{\widehat{A}(TZ)\}^{(12)}\\
&-(m-32)\{\widehat{A}(TZ)\mathrm{ch}({F}_\CC)\}^{(12)}\\
=&(p_1(TZ)-p_1(F))\left\{-\frac{e^{\frac{1}{24}(p_1(TZ)-p_1(F))}-1}{p_1(TZ)-p_1(F)}
\widehat{A}(TZ)\mathrm{ch}(\mathfrak{F})+
e^{\frac{1}{24}(p_1(TZ)-p_1(F))}\widehat{A}(TZ)\right\}^{(8)},
\end{split}\ee
where
\be \begin{split} \mathfrak{F}=&\wedge^2{F}_\CC+T_\CC Z+
\frac{(m-32)(m-31)}{2}-2-(m-32){F}_\CC.
\end{split} \ee

\end{corollary}

Putting $m=32$ in the above corollary,  we get
\begin{corollary}If $\dim F=32$, the following identity holds,
\be \begin{split} &\{\widehat{A}(TZ)e^{\frac{c}{2}}\mathrm{ch}(\wedge^2{F}_\CC)\}^{(12)}+
\{\widehat{A}(TZ)e^{\frac{c}{2}}\mathrm{ch}(T_\CC Z)\}^{(12)}-2\{\widehat{A}(TZ)e^{\frac{c}{2}}\}^{(12)}\\
&+5\{\widehat{A}(TZ)e^{\frac{c}{2}}\mathrm{ch}(\widetilde{\xi_\CC}\otimes\widetilde{\xi_\CC})\}^{(12)}
+3\{\widehat{A}(TZ)e^{\frac{c}{2}}\mathrm{ch}((1-{F}_\CC)\otimes\widetilde{\xi_\CC})\}^{(12)}\\
=&(p_1(TZ)-p_1(F))\\
&\cdot\left\{-\frac{e^{\frac{1}{24}(p_1(TZ)-p_1(F))}-1}{p_1(TZ)-p_1(F)}
\widehat{A}(TZ)e^{\frac{c}{2}}\mathrm{ch}(\mathfrak{G})+
e^{\frac{1}{24}(p_1(TZ)-p_1(F))}\widehat{A}(TZ)e^{\frac{c}{2}}\right\}^{(8)},
\end{split}\ee
where
\be \begin{split} \mathfrak{G}=&\wedge^2{F}_\CC+T_\CC Z
-2+5\widetilde{\xi_\CC}\otimes\widetilde{\xi_\CC}+
3(1-{F}_\CC)\otimes\widetilde{\xi_\CC};
\end{split} \ee
if  $\xi$ is trivial, we obtain the Green-Schwarz formula (0.3),
\be \begin{split} &\{\widehat{A}(TZ)\mathrm{ch}(\wedge^2{F}_\CC)\}^{(12)}+\{\widehat{A}(TZ)\mathrm{ch}(T_\CC Z\}^{(12)}-2\{\widehat{A}(TZ)\}^{(12)}\\
=&(p_1(TZ)-p_1(F))\left\{-\frac{e^{\frac{1}{24}(p_1(TZ)-p_1(F))}-1}{p_1(TZ)-p_1(F)}
\widehat{A}(TZ)\mathrm{ch}(\wedge^2{F}_\CC+T_\CC Z-2)+
e^{\frac{1}{24}(p_1(TZ)-p_1(F))}\widehat{A}(TZ)\right\}^{(8)}.
\end{split}\ee
\end{corollary}

\section{Derivation of the Green-Schwarz Type Factorizations from Modularity}
In this section, we will derive the Green-Schwarz type
factorization formulas presented in Section 1 via the modularity
of   $\mathcal{P}_2(TZ, F_1-F_2, \xi, \tau)$.
\subsection{Preliminaries}In this subsection, we recall
some basic knowledge about the Jacobi theta functions, modular forms and Eisenstein series.
 Although we will not use all the things recalled here, we still put them in this subsection for completeness.

Let $ SL_2(\mathbf{Z}):= \left\{\left.\left(\begin{array}{cc}
                                      a&b\\
                                      c&d
                                     \end{array}\right)\right|a,b,c,d\in\mathbf{Z},\ ad-bc=1
                                     \right\}
                                     $
 as usual be the modular group. Let
$S=\left(\begin{array}{cc}
      0&-1\\
      1&0
\end{array}\right)$, $  T=\left(\begin{array}{cc}
      1&1\\
      0&1
\end{array}\right)$
be the two generators of $ SL_2(\mathbf{Z})$. Their actions on
$\mathbf{H}$ are given by $ S:\tau\rightarrow-\frac{1}{\tau}, \ \
\ T:\tau\rightarrow\tau+1.$

The four Jacobi theta functions are defined as follows (cf.
\cite{C}): \h \theta(v,\tau)=2q^{1/8}\sin(\pi v)
\prod_{j=1}^\infty\left[(1-q^j)(1-e^{2\pi \sqrt{-1}v}q^j)(1-e^{-2\pi
\sqrt{-1}v}q^j)\right]\ ,\e \h \theta_1(v,\tau)=2q^{1/8}\cos(\pi
v)
 \prod_{j=1}^\infty\left[(1-q^j)(1+e^{2\pi \sqrt{-1}v}q^j)
 (1+e^{-2\pi \sqrt{-1}v}q^j)\right]\ ,\e
\h \theta_2(v,\tau)=\prod_{j=1}^\infty\left[(1-q^j)
 (1-e^{2\pi \sqrt{-1}v}q^{j-1/2})(1-e^{-2\pi \sqrt{-1}v}q^{j-1/2})\right]\
 ,\e
\h \theta_3(v,\tau)=\prod_{j=1}^\infty\left[(1-q^j) (1+e^{2\pi
\sqrt{-1}v}q^{j-1/2})(1+e^{-2\pi \sqrt{-1}v}q^{j-1/2})\right]\ .\e
They are all holomorphic functions for $(v,\tau)\in \mathbf{C \times
H}$, where $\mathbf{C}$ is the complex plane and $\mathbf{H}$ is the
upper half plane.

When acted by $S$ and $T$, the theta functions obey the following
transformation laws (cf. \cite{C}), \be \theta(v,\tau+1)=e^{\pi
\sqrt{-1}\over 4}\theta(v,\tau),\ \ \
\theta\left(v,-{1}/{\tau}\right)={1\over\sqrt{-1}}\left({\tau\over
\sqrt{-1}}\right)^{1/2} e^{\pi\sqrt{-1}\tau v^2}\theta\left(\tau
v,\tau\right)\ ;\ee \be \theta_1(v,\tau+1)=e^{\pi \sqrt{-1}\over
4}\theta_1(v,\tau),\ \ \
\theta_1\left(v,-{1}/{\tau}\right)=\left({\tau\over
\sqrt{-1}}\right)^{1/2} e^{\pi\sqrt{-1}\tau v^2}\theta_2(\tau
v,\tau)\ ;\ee \be\theta_2(v,\tau+1)=\theta_3(v,\tau),\ \ \
\theta_2\left(v,-{1}/{\tau}\right)=\left({\tau\over
\sqrt{-1}}\right)^{1/2} e^{\pi\sqrt{-1}\tau v^2}\theta_1(\tau
v,\tau)\ ;\ee \be\theta_3(v,\tau+1)=\theta_2(v,\tau),\ \ \
\theta_3\left(v,-{1}/{\tau}\right)=\left({\tau\over
\sqrt{-1}}\right)^{1/2} e^{\pi\sqrt{-1}\tau v^2}\theta_3(\tau
v,\tau)\ .\ee

\begin{definition} Let $\Gamma$ be a subgroup of $SL_2(\mathbf{Z}).$ A modular form over $\Gamma$ is
a holomorphic function $f(\tau)$ on $\mathbf{H}\cup
\{\infty\}$ such that for any
 $ g=\left(\begin{array}{cc}
             a&b\\
             c&d
             \end{array}\right)\in\Gamma\ $, the following property
             holds,
 $$f(g\tau):=f\left(\frac{a\tau+b}{c\tau+d}\right)=\chi(g)(c\tau+d)^lf(\tau), $$
 where $\chi:\Gamma\rightarrow\mathbf{C}^*$ is a character of
 $\Gamma$ and $l$ is called the weight of $f$.
 \end{definition}

Let
\be E_{2k}=1-\frac{4k}{B_{2k}}\sum_{n=1}^{\infty}\left(\underset{d|n}{\sum}d^{2k-1}\right)q^n \ee
be the Eisenstein series, where $B_{2k}$ is the $2k$-th Bernoulli number.

When $k>1$, $E_{2k}$ is a modular form of weight $2k$ over
$SL_2(\mathbf{Z})$. However, unlike other Eisenstein series,
$E_2(\tau)=1-24\sum_{n=1}^{\infty}\left(\underset{d|n}{\sum}d\right)q^n=1-24q-72q^2-96q^3-\cdots$ is not a modular form over $SL(2,\ZZ)$, instead it is
a quasimodular form over $SL(2,\ZZ)$ satisfying \be
E_2\left(\frac{a\tau+b}{c\tau+d}\right)=(c\tau+d)^2E_2(\tau)-\frac{6\ii
c(c\tau+d)}{\pi}. \ee In particular, we have \be
E_2(\tau+1)=E_2(\tau),\ee \be
E_2\left(-\frac{1}{\tau}\right)=\tau^2E_2(\tau)-\frac{6\ii\tau}{\pi}.\ee
For the precise definition of quasimodular forms, see \cite{KZ}.


Let $ \Gamma_0(2)=\left\{\left.\left(\begin{array}{cc}
a&b\\
c&d
\end{array}\right)\in SL_2(\mathbf{Z})\right|c\equiv0\ \ (\rm mod \ \ 2)\right\}$,
$ \Gamma^0(2)=\left\{\left.\left(\begin{array}{cc}
a&b\\
c&d
\end{array}\right)\in SL_2(\mathbf{Z})\right|b\equiv0\ \ (\rm mod \ \ 2)\right\}$
be the two modular subgroups of $SL_2(\mathbf{Z})$. It is known
that the generators of $\Gamma_0(2)$ are $T,ST^2ST$ and the
generators of $\Gamma^0(2)$ are $STS,T^2STS$ (cf. \cite{C}).

Consider the $q$-series:
\be \delta_1(\tau)=\frac{1}{4}+6\sum_{n=1}^{\infty}\underset{d\ odd}{\underset{d|n}{\sum}}dq^n={1\over 4}+6q+6q^2+\cdots,\ee
\be \varepsilon_1(\tau)=\frac{1}{16}+\sum_{n=1}^{\infty}\underset{d|n}{\sum}(-1)^dd^3q^n={1\over
16}-q+7q^2+\cdots,\ee
\be \delta_2(\tau)=-\frac{1}{8}-3\sum_{n=1}^{\infty}\underset{d\ odd}{\underset{d|n}{\sum}}dq^{n/2}=-{1\over 8}-3q^{1/2}-3q-\cdots,\ee
\be \varepsilon_2(\tau)=\sum_{n=1}^{\infty}\underset{n/d\ odd}{\underset{d|n}{\sum}}d^3q^{n/2}=q^{1/2}+8q+\cdots.\ee

Simply writing
$\theta_j=\theta_j(0,\tau),\ 1\leq j \leq 3,$ then we have (cf. \cite{HBJ} and \cite{Liu3}),
$$ \delta_1(\tau)=\frac{1}{8}(\theta_2^4+\theta_3^4), \ \ \ \
\varepsilon_1(\tau)=\frac{1}{16}\theta_2^4 \theta_3^4\ ,$$
$$\delta_2(\tau)=-\frac{1}{8}(\theta_1^4+\theta_3^4), \ \ \ \
\varepsilon_2(\tau)=\frac{1}{16}\theta_1^4 \theta_3^4\ .$$

If $\Gamma$ is a modular subgroup, let
$M_\mathbf{R}(\Gamma)$ denote the ring of modular forms
over $\Gamma$ with real Fourier coefficients.
\begin{lemma} [\protect cf. \cite{Liu1}] One has that $\delta_1(\tau)\ (resp.\ \varepsilon_1(\tau) ) $
is a modular form of weight $2 \ (resp.\ 4)$ over $\Gamma_0(2)$,
$\delta_2(\tau) \ (resp.\ \varepsilon_2(\tau))$ is a modular form
of weight $2\ (resp.\ 4)$ over $\Gamma^0(2)$ and moreover
$M_\mathbf{R}(\Gamma^0(2))=\mathbf{R}[\delta_2(\tau),
\varepsilon_2(\tau)]$. Moreover, we have
transformation laws \be
\delta_2\left(-\frac{1}{\tau}\right)=\tau^2\delta_1(\tau),\ \ \ \ \
\ \ \ \ \
\varepsilon_2\left(-\frac{1}{\tau}\right)=\tau^4\varepsilon_1(\tau).\ee

\end{lemma}

\subsection{The modular form $\mathcal{P}_2(TZ, F_1-F_2, \xi, \tau)$} 
Let $F $ (resp. $G$) be a Hermitian vector bundle over $X$
equipped with a Hermitian connection $\nabla^F$ (resp.
$\nabla^G$). For any complex number $t$, let
$$\Lambda_t(F)=\mathbf{C}|_X+tF+t^2\Lambda^2(F)+\cdots ,
\ \  S_t(F)=\mathbf{C}|_X+tF+t^2S^2(F)+\cdots$$  denote
respectively the total exterior and symmetric powers  of $F$,
which live in $K(X)[[t]].$ The following relations between these
two operations  hold (cf. \cite[Chap. 3]{At}), \be
S_t(F)=\frac{1}{\Lambda_{-t}(F)},\ \ \ \
 \Lambda_t(F-G)=\frac{\Lambda_t(F)}{\Lambda_t(G)}.\ee

The connections $\nabla^F, \nabla^G$ naturally induce connections
on $\Lambda_t(F), S_t(F) $ etc. Moreover, if $\{\omega_i \}$,
$\{{\omega_j}' \}$ are formal Chern roots for the Hermitian vector
bundles $F$, $G$ respectively, then (cf. \cite[Chap. 1]{Hir}), \be
\mathrm{ch}\left(\Lambda_t{(F)},
\nabla^{\Lambda_t(F)}\right)=\prod\limits_i(1+e^{\omega_i}t)\ee
and we have the following formulas for Chern character forms,
\be{\rm ch}\left(S_t(F), \nabla^{S_t(F)} \right)=\frac{1}{{\rm
ch}\left(\Lambda_{-t}(F),\nabla^{\Lambda_{-t}(F)}
\right)}=\frac{1}{\prod\limits_i (1-e^{\omega_i}t)}\ ,\ee \be{\rm
ch}\left(\Lambda_t(F-G), \nabla^{\Lambda_t(F-G)}
\right)=\frac{{\rm ch}\left(\Lambda_{t}(F),\nabla^{\Lambda_t(F)}
\right)}{{\rm ch}\left(\Lambda_{t}(G),\nabla^{\Lambda_t(G)}
\right)}=\frac{\prod\limits_i(1+e^{\omega_i}t)}{\prod\limits_j(1+e^{{\omega_j}'}t)}\
.\ee

Let $q=e^{2\pi \sqrt{-1}\tau}$ with $\tau \in \mathbf{H}$, the upper
half complex plane.

Set
\be
\begin{split} \Theta_2(T_\CC Z, {F_1}_\CC-{F_2}_\CC, \xi_\CC)=&\bigotimes_{u=1}^\infty
S_{q^u}(\widetilde{T_\CC Z}) \otimes \bigotimes_{v=1}^\infty
\Lambda_{-q^{v-{1\over 2}}}(\widetilde{{F_1}_\CC}-\widetilde{{F_2}_\CC}-2\widetilde{\xi_\CC})\\
&\otimes\bigotimes_{r=1}^{\infty}\Lambda_{q^{r-\frac{1}{2}}}(\widetilde{\xi_\CC})
\otimes\bigotimes_{s=1}^{\infty}\Lambda_{q^{s}}(\widetilde{\xi_\CC}) ,\\
\end{split}\ee which is
an element in $K(X)[[q^{1\over2}]]$.

Clearly, $\Theta_2(T_\CC Z, {F_1}_\CC-{F_2}_\CC, \xi_\CC)$ admits
a formal Fourier expansion in $q^{1/2}$ as \be \Theta_2(T_\CC Z,
{F_1}_\CC-{F_2}_\CC, \xi_\CC)=B_0+B_1q^{1/2}+B_2q\cdots,\ee where
the $B_j$'s are elements in the semi-group formally generated by
complex vector bundles over $X$. Moreover, they carry canonically
induced connections denoted by $\nabla^{B_j}$ and let
$\nabla^{\Theta_2}$ be the induced connections with
$q^{1/2}$-coefficients on $\Theta_2$.

Set
\be \begin{split} &\mathcal{P}_2(TZ, F_1-F_2, \xi, \tau)\\
:=&\left\{e^{\frac{1}{24}E_2(\tau)(p_1(TZ)-p_1(F_1)+p_1(F_2))}\widehat{A}(TZ)\cosh\left(\frac{c}{2}\right)\mathrm{ch}\left(\Theta_2(T_{\mathbf C}Z,{F_1}_\CC-{F_2}_\CC, \xi_\CC)\right)\right\}^{(12)}.
\end{split}\ee

\begin{proposition} $\mathcal{P}_2(TZ, F_1-F_2, \xi, \tau)$ is a modular form of weight $6$ over $\Gamma^0(2)$.

\end{proposition}

\begin{proof} Let $\{\pm 2\pi \ii y_k\}$ (resp. $\{\pm 2\pi \ii z_k\}$, $\{\pm 2\pi \ii x_j\}$) be the
 formal Chern roots for $({F_1}_\CC, \nabla^{{F_1}_\CC})$ (resp. $({F_2}_\CC, \nabla^{{F_2}_\CC})$,
 $(TZ_\CC, \nabla^{TZ_\CC})$). Let $c=2\pi \ii u$.
By the Chern root algorithm, we have
\be \begin{split} &\mathcal{P}_2(TZ, F_1-F_2, \xi, \tau)\\
=& \left\{e^{\frac{1}{24}E_2(\tau)(p_1(TZ)-p_1(F_1)+p_1(F_2))}\widehat{A}(TZ)\cosh\left(\frac{c}{2}\right)\mathrm{ch}\left(\Theta_2(T_\CC Z, {F_1}_\CC-{F_2}_\CC, \xi_\CC)\right)\right\}^{(12)}\\
=&\left\{e^{\frac{1}{24}E_2(\tau)(p_1(TZ)-p_1(F_1)+p_1(F_2))}\left(\prod_{j=1}^{5}\left(x_j\frac{\theta'(0,\tau)}{\theta(x_j,\tau)}\right)
\right)\left(\prod_{j=1}^{\left[\frac{m}{2}\right]}\frac{\theta_{2}(y_j,\tau)}{\theta_{2}(0,\tau)}\right)\right.\\
& \left.\cdot\left(\prod_{k=1}^{\left[\frac{n}{2}\right]}\frac{\theta_{2}(0,\tau)}{\theta_{2}(z_k,\tau)}\right)
\cdot\frac{\theta_2^2(0,\tau)}{\theta_2^2(u,\tau)}\frac{\theta_3(u,\tau)}{\theta_3(0,\tau)}\frac{\theta_1(u,\tau)}{\theta_1(0,\tau)}\right\}^{(12)}.\end{split}\ee

Then we can apply the transformation laws (2.1)-(2.4) for theta
functions as well as the transformation laws (2.7), (2.8) to
(2.21) to get the desired results.
\end{proof}

 In addition to the above modular form, we have also
constructed in \cite{HLZ} the modular form
\be \begin{split}\mathcal{ P}_1(TZ, V, \xi, \tau):=&\left\{e^{\frac{1}{24}E_2(\tau)(p_1(TZ)-p_1(V))}\right.\\
&\ \ \left. \cdot \frac{\widehat{A}(TZ)\mathrm{det}^{1/2}\left(2\cosh\left(\frac{\ii}{4\pi}R^V\right)\right)}{\cosh^2\left(\frac{c}{2}\right)}
\mathrm{ch}\left(\Theta_1(T_{\mathbf C}Z,V_{\mathbf C}, \xi_\CC)\right)\right\}^{(12)},\end{split}\ee
where
\be \begin{split} \Theta_1(T_\CC Z, V_\CC, \xi_\CC)=&\bigotimes_{u=1}^\infty S_{q^u}(\widetilde{T_\CC Z})
\otimes \bigotimes_{v=1}^\infty \Lambda_{q^v}(\widetilde{V_\CC}-2\widetilde{\xi}_\CC)\\
&\otimes \bigotimes_{r=1}^\infty
\Lambda_{q^{r-1/2}}(\widetilde{\xi}_\CC)\otimes
\bigotimes_{s=1}^\infty
\Lambda_{-q^{s-1/2}}(\widetilde{\xi}_\CC).\end{split} \ee We
showed in \cite{HLZ} that $\mathcal{P}_1(TZ, V, \xi, \tau)$ is a
modular form of weight $6$ over $\Gamma_0(2)$ while
$\mathcal{P}_2(TZ, V, \xi, \tau)$ is a modular form of weight $6$
over $\Gamma^0(2)$ and moreover they are modularly related in the
sense that
$$\mathcal{ P}_1\left(TZ, V, \xi,
-\frac{1}{\tau}\right)=2^{[\frac{\mathrm{dim}
V}{2}]}\tau^6\mathcal{ P}_2(TZ, V, \xi, \tau).$$ We call such a
pair of modular forms a modular pair (see \cite{HLZ} for the cases
of general dimensions).

One can use this modular pair $(\mathcal{P}_1(TZ, V, \xi, \tau),
\mathcal{P}_2(TZ, V, \xi, \tau)) $ to derive the
Alvarez-Gaum$\mathrm{\acute{e}}$-Witten miraculous anomaly
cancellation formula by setting $V=TZ,\, \xi=\CC$ and obtain its
various generalizations (see \cite{Liu1, Liu3, HZ2, HL, HLZ, LW}
for details).

In the following subsection, we will use the modularity of
$\mathcal{P}_2(TZ, F_1-F_2, \xi, \tau)$ to derive the
Green-Schwarz type factorization formulas.
 It's amazing to see that all these anomaly cancellations  due to Alvarez-Gaum$\mathrm{\acute{e}}$-Witten,
 Green-Schwarz as well as Schwarz-Witten
can be derived from the modular pair $(\mathcal{P}_1(TZ, V, \xi,
\tau), \mathcal{P}_2(TZ, V, \xi, \tau)) $.

\subsection{Derivation of Green-Schwarz type factorizations from modularity}


From Proposition 2.1, we see that $\mathcal{P}_2(TZ, F_1-F_2, \xi, \tau)$ is a modular form of weight $6$ over $\Gamma^0(2)$. Therefore, by Lemma 2.1, there exist $h_1, h_2\in \Omega^{12}(X)$ such that
 \be \mathcal{P}_2(TZ, F_1-F_2, \xi, \tau)=h_0(8\delta_2)^3+h_1(8\delta_2)\varepsilon_2.
\ee
Therefore
\be \begin{split} &\left\{e^{\frac{1}{24}(1-24q+\cdots)(p_1(TZ)-p_1(F_1)+p_1(F_2))}\widehat{A}(TZ)\cosh\left(\frac{c}{2}\right)\mathrm{ch}
\left(B_0+B_1q^{\frac{1}{2}}+B_2q+\cdots\right)\right\}^{(12)}\\
=&h_0(8\delta_2)^3+h_1(8\delta_2)\varepsilon_2\\
=&h_0(-1-24q^{\frac{1}{2}}-24q-\cdots)^3+h_1(-1-24q^{\frac{1}{2}}-24q-\cdots)(q^{\frac{1}{2}}+8q+\cdots).
\end{split}\ee

Comparing the coefficients of $1,\, q^{\frac{1}{2}}$ and $q$ in
both sides of (2.25), we have \be
\left\{e^{\frac{1}{24}(p_1(TZ)-p_1(F_1)+p_1(F_2))}\widehat{A}(TZ)\cosh\left(\frac{c}{2}\right)\mathrm{ch}
(B_0)\right\}^{(12)}=-h_0,\ee \be
\left\{e^{\frac{1}{24}(p_1(TZ)-p_1(F_1)+p_1(F_2))}\widehat{A}(TZ)\cosh\left(\frac{c}{2}\right)\mathrm{ch}
(B_1)\right\}^{(12)}=-h_1-72h_0,\ee \be
\begin{split} &\left\{e^{\frac{1}{24}(p_1(TZ)-p_1(F_1)+p_1(F_2))}(-(p_1(TZ)+p_1(F_1)-p_1(F_2)))
\widehat{A}(TZ)\cosh\left(\frac{c}{2}\right)\mathrm{ch}(B_0)\right.\\
&\left.+e^{\frac{1}{24}(p_1(TZ)-p_1(F_1)+p_1(F_2))}\widehat{A}(TZ)\cosh\left(\frac{c}{2}\right)\mathrm{ch}(B_2)\right\}^{(12)}\\
=&-32h_1-1800h_0.\end{split}
\ee

By (2.26)-(2.28), we see that \be
\begin{split} &\left\{e^{\frac{1}{24}(p_1(TZ)-p_1(F_1)+p_1(F_2))}(-(p_1(TZ)+p_1(F_1)-p_1(F_2)))
\widehat{A}(TZ)\cosh\left(\frac{c}{2}\right)\mathrm{ch}(B_0)\right.\\
&\left.+e^{\frac{1}{24}(p_1(TZ)-p_1(F_1)+p_1(F_2))}\widehat{A}(TZ)\cosh\left(\frac{c}{2}\right)\mathrm{ch}(B_2)\right\}^{(12)}\\
=&\left\{e^{\frac{1}{24}(p_1(TZ)-p_1(F_1)+p_1(F_2))}\widehat{A}(TZ)\cosh\left(\frac{c}{2}\right)\mathrm{ch}(32B_1-504B_0)
\right\}^{(12)}.\end{split}
\ee

In the following, let's expand $\Theta_2(T_\CC Z,
{F_1}_\CC-{F_2}_\CC, \xi_\CC)$ to find $B_0,\, B_1,\, B_2$. In
fact, we have
\be \begin{split}&\Theta_2(T_\CC Z, {F_1}_\CC-{F_2}_\CC, \xi_\CC) \\
=&\bigotimes_{u=1}^\infty
S_{q^u}(\widetilde{T_\CC Z}) \otimes \bigotimes_{v=1}^\infty
\Lambda_{-q^{v-{1\over 2}}}(\widetilde{{F_1}_\CC}-\widetilde{{F_2}_\CC}-2\widetilde{\xi_\CC})\\
&\otimes\bigotimes_{r=1}^{\infty}\Lambda_{q^{r-\frac{1}{2}}}(\widetilde{\xi_\CC})
\otimes\bigotimes_{s=1}^{\infty}\Lambda_{q^{s}}(\widetilde{\xi_\CC})\\
=&\bigotimes_{u=1}^\infty
S_{q^u}(\widetilde{T_\CC Z}) \\
&\otimes \bigotimes_{v=1}^\infty
\Lambda_{-q^{v-{1\over 2}}}(\widetilde{{F_1}_\CC})\\
&\otimes \frac{1}{\bigotimes_{v=1}^\infty
\Lambda_{-q^{v-{1\over 2}}}(\widetilde{{F_2}_\CC})} \\
&\otimes  \frac{1}{\bigotimes_{v=1}^\infty
(\Lambda_{-q^{v-{1\over 2}}}(\widetilde{\xi_\CC}))^2}\\
&\otimes\bigotimes_{r=1}^{\infty}\Lambda_{q^{r-\frac{1}{2}}}(\widetilde{\xi_\CC})\\
&\otimes\bigotimes_{s=1}^{\infty}\Lambda_{q^{s}}(\widetilde{\xi_\CC})\\
=&(1+(T_\CC Z-10)q+O(q^2))\\
&\otimes \left(1+(m-{F_1}_\CC)q^{\frac{1}{2}}+\left(\wedge^2{F_1}_\CC-m{F_1}_\CC+\frac{m(m+1)}{2}\right)q+O(q^{\frac{3}{2}})\right)\\
&\otimes \left(1+({F_2}_\CC-n)q^{\frac{1}{2}}+\left(S^2{F_2}_\CC-n{F_2}_\CC+\frac{n(n-1)}{2}\right)q+O(q^{\frac{3}{2}})\right)\\
&\otimes \left(1+2\widetilde{\xi_\CC}q^{\frac{1}{2}}+\left(3\widetilde{\xi_\CC}\otimes \widetilde{\xi_\CC}+4\widetilde{\xi_\CC}\right)q+O(q^{\frac{3}{2}})\right)\\
&\otimes \left(1+\widetilde{\xi_\CC}q^{\frac{1}{2}}-2\widetilde{\xi_\CC}q+O(q^{\frac{3}{2}})\right)\\
&\otimes(1+\widetilde{\xi_\CC}q+O(q^2))\\
=&1+(m-{F_1}_\CC+{F_2}_\CC-n+3\widetilde{\xi_\CC})q^{\frac{1}{2}}\\
&+\left(\wedge^2{F_1}_\CC+S^2{F_2}_\CC- {F_1}_\CC\otimes{F_2}_\CC+T_\CC Z+\frac{(m-n)^2+(m-n)}{2}-10-(m-n)( {F_1}_\CC- {F_2}_\CC)\right.\\
 &\left. \ \ \ \ \ \ \ \ +5\widetilde{\xi_\CC}\otimes \widetilde{\xi_\CC}+3(m-{F_1}_\CC+{F_2}_\CC-n+1)\otimes\widetilde{\xi_\CC}\right)q+O(q^{\frac{3}{2}}).
\end{split}\ee

Therefore, we have
\be \begin{split} B_0=&1,\\
B_1=&m-{F_1}_\CC+{F_2}_\CC-n+3\widetilde{\xi_\CC},\\
B_2=&\wedge^2{F_1}_\CC+S^2{F_2}_\CC- {F_1}_\CC\otimes{F_2}_\CC+T_\CC Z+\frac{(m-n)^2+(m-n)}{2}-10-(m-n)( {F_1}_\CC- {F_2}_\CC)\\
&+5\widetilde{\xi_\CC}\otimes \widetilde{\xi_\CC}+3(m-{F_1}_\CC+{F_2}_\CC-n+1)\otimes\widetilde{\xi_\CC}\end{split} \ee

From (2.29), we see that
\be
\begin{split} &\left\{\widehat{A}(TZ)\cosh\left(\frac{c}{2}\right)\mathrm{ch}(B_2-32B_1+504B_0)
\right\}^{(12)}\\
=&(p_1(TZ)-p_1(F_1)+p_1(F_2))\\
&\cdot\left\{-\frac{e^{\frac{1}{24}(p_1(TZ)-p_1(F_1)+p_1(F_2))}-1}{p_1(TZ)-p_1(F_1)+p_1(F_2)}
\widehat{A}(TZ)e^{\frac{c}{2}}\mathrm{ch}(B_2-32B_1+504B_0)\right.\\
&\ \ \ \ \ \ \
\left.+e^{\frac{1}{24}(p_1(TZ)-p_1(F_1)+p_1(F_2))}\widehat{A}(TZ)e^{\frac{c}{2}}\mathrm{ch}(B_0)\right\}^{(8)}.\end{split}
\ee

However, from (2.31), we have
\be \begin{split} &B_2-32B_1+504B_0\\
 =&\wedge^2{F_1}_\CC+S^2{F_2}_\CC-{F_1}_\CC\otimes {F_2}_\CC+T_\CC Z+
\frac{(m-n-32)(m-n-31)}{2}-2\\
&-(m-n-32)({F_1}_\CC-{F_2}_\CC)+5\widetilde{\xi_\CC}\otimes\widetilde{\xi_\CC}+
3(m-n-31-{F_1}_\CC+{F_2}_\CC)\otimes\widetilde{\xi_\CC}. \end{split}\ee

 Theorem 1.1 follows from (2.32) and (2.33).

$$ $$
\noindent {\bf Acknowledgements.} The first author is partially
supported by a start-up grant from National University of
Singapore. The second author is partially supported by NSF. The
third author is partially supported by  MOE and NNSFC.

\end{document}